\documentclass[12pt,letterpaper]{amsart}
\usepackage{geometry,mathpazo}
\usepackage{amsmath,amssymb,amsfonts,amsthm,url,setspace}
\newcommand{\norm}[1]{\Vert#1\Vert}
\newtheorem{theorem}{Theorem}
\newtheorem*{theorem*}{Theorem}
\newtheorem{lemma}{Lemma}
\newtheorem{corollary}{Corollary}
\theoremstyle{remark}
\newtheorem{remark}{Remark}
\newtheorem{example}{Example}

\newcommand{\C}{\mathbb{C}}

\newcommand{\D}{\Omega}
\newcommand{\Dc}{\overline{\Omega}}
\newcommand{\dbar}{\overline{\partial}}
\usepackage{hyperref}
\onehalfspace
\title{On compactness of the $\dbar$-Neumann problem and Hankel operators}

\author{Mehmet \c{C}el\.ik}
\address[Mehmet \c{C}elik]{University of North Texas at Dallas, Department of
Mathematics and Information Sciences,  7300 Houston School Road, 
Dallas, TX 75241}
\email{Mehmet.Celik@unt.edu}

\author{S\"{o}nmez \c{S}ahuto\u{g}lu}
\address[S\"{o}nmez \c{S}ahuto\u{g}lu]{ University of Toledo, Department of
Mathematics \& Statistics, 2801 W. Bancroft, Toledo, OH 43606, USA}
\email{sonmez.sahutoglu@utoledo.edu}

\thanks{The second author is supported in part by University of Toledo's 
Summer Research Awards and Fellowships Program}

\date{\today}
\begin{document}
\begin{abstract}
Let $\D=\D_1\setminus \Dc_2$, where $\D_1$ and $\D_2$ are two smooth bounded
pseudoconvex domains in $\C^n, n\geq 3,$ such that $\Dc_2\subset \D_1.$ Assume
that the $\dbar$-Neumann operator of $\D_1$ is compact and  the interior of the
Levi-flat points in the boundary of $\D_2$ is not empty (in the relative
topology). Then we show that the Hankel operator on $\D$ with symbol 
$\phi, H^{\D}_{\phi},$ is compact for every $\phi\in C(\Dc)$ but the
$\dbar$-Neumann operator on $\D$ is not compact.
\end{abstract}

\maketitle

Let $\D$ be a domain in $\C^n$ and $A^2(\D)$ denote the Bergman space on $\D$,
the space of square integrable holomorphic functions on $\D.$ The Bergman
projection, the orthogonal projection from $L^2(\D)$ onto $A^2(\D),$ is denoted 
by $P^{\D}$ and the Hankel operator  with symbol $\phi\in L^{\infty}(\D),$
denoted by $H^{\D}_{\phi},$ is defined as 
$H^{\D}_{\phi}(f)=\phi f-P^{\D}(\phi f)$ for $f\in A^2(\D).$

The $\dbar$-Neumann problem is  solving $\Box u=v$ where
$\Box=\dbar\dbar^*+\dbar^*\dbar,$ on square integrable $(0,1)$-forms and 
$\dbar^*$ is the Hilbert space adjoint of $\dbar$. We will denote the solution
operator to $\Box$ on a domain $\D$, the $\dbar$-Neumann operator on $\D$, by
$N^{\D}$. On bounded pseudoconvex domains, H\"ormander \cite{Hormander65}
showed that $N$ is a bounded operator on $L_{(0,1)}^{2}(\D)$, and Kohn
\cite{Kohn63} showed that $P^{\D}=I-\dbar^* N^{\D}\dbar.$ Therefore,
$H^{\D}_{\phi}(f)=\dbar^*N^{\D}(f\dbar\phi)$ for $f\in A^2(\D)$ and $\phi\in
C^1(\Dc).$ We refer the reader to \cite{ChenShawBook,StraubeBook} for more
information about the $\dbar$-Neumann problem and to \cite{CuckovicSahutoglu09}
(and references therein) for  more information on compactness of Hankel
operators on Bergman spaces.  

Given Kohn's formula, it is natural to expect strong connections between 
$N^{\D}$ and Hankel operators on $A^2(\D).$  For example, if $\D$ is bounded and
pseudoconvex, and $N^{\D}$ is compact on $L^2_{(0,1)}(\D)$, then 
$H_{\phi}^{\D}$ is compact on $A^2(\D)$ for all $\phi\in C(\Dc)$ (see 
\cite[Theorem 3]{Haslinger08} and \cite[Proposition 4.1]{StraubeBook}). We are
interested in the converse, which is a question of Fu and Straube \cite[Remark
2]{FuStraube01}: does compactness of $H^{\D}_{\phi}$ on $A^2(\D)$ for all
$\phi\in C(\Dc)$ imply that $N^{\D}$ is compact on $L^2_{(0,1)}(\D)$?  The
answer to this question is still open in general. However, if $\D$ is allowed to
be non-pseudoconvex  we can show that the answer is no (see Theorem
\ref{ThmMain} below).

We call $\D$ an annulus type domain if $\D=\D_1\setminus \Dc_2$ where $\D_1$ 
and $\D_2$ are smooth, bounded,  pseudoconvex, and $\Dc_2\subset \D_1.$
The following theorem of Shaw, contained in \cite[Theorem 3.5]{Shaw10},
guarantees that the $\dbar$-Neumann operator exists on  annulus
type domains in $\C^n$ for $n\geq 3$ and it is connected to the Bergman
projection the same way as it is on bounded pseudoconvex domains.

\begin{theorem*}[Shaw]
 Let $\D=\D_1\setminus \Dc_2$, where $\D_1$ and $\D_2$ are two smooth bounded
pseudoconvex domains in $\C^n,n\geq 3,$ such that $\Dc_2\subset
\D_1.$ Then 
\begin{itemize}
 \item[i.] $N^{\D}_{q}$ exists on $L^2_{(0,q)}(\D)$ for  $1\leq q\leq n-2,$ 
\item[ii.] $\dbar^*N^{\D}$ is the canonical solution operator for $\dbar$,
\item[iii.] $P^{\D}=I-\dbar^*N^{\D}\dbar$.
\end{itemize}
\end{theorem*}

In fact Shaw (\cite[Theorem 3.5]{Shaw10}) showed that the $\dbar$-Neumann
operator is bounded on $(0,1)$-forms for $n\geq 2.$ However, the space of
harmonic forms $\mathcal{H}^{\D}_{(0,1)},$ defined in the next section, is
infinite dimensional when $n=2$ and trivial when $n\geq 3.$ Hence, when $n=2$
items ii. and iii. in Shaw's theorem above are not valid.  

The following theorem is our main result. We note that $B(p,r)$ denotes the open
ball centered at $p$ with radius $r,$ and a point $p,$ in the boundary of a
smooth domain $\D\subset \C^n$, is called Levi-flat if the Levi form of $\D,$
the restriction of the complex Hessian of a defining function onto complex
tangent space, is constant zero at $p.$ We denote the boundary of a domain $\D$
by $b\D.$

\begin{theorem}\label{ThmMain}
Let $\D=\D_1\setminus \Dc_2$, where $\D_1$ and $\D_2$ are two smooth bounded
pseudoconvex domains in $\C^n, n\geq 3,$ such that $\Dc_2\subset \D_1.$ Assume
that the $\dbar$-Neumann operator $N ^{\D_1}$ is compact on 
$L_{(0,1)}^{2}(\D_1)$ and  that  there exists a ball, $B(p,r)$ centered   $p\in
b\D_2$ with radius $r>0$ such that $B(p,r)\cap b\D_2$ is a Levi-flat surface.
Then the Hankel operator $H^{\D}_{\phi}$ is compact on $A^{2}(\D)$ 
for every $\phi\in C(\Dc)$ but the $\dbar$-Neumann operator
$N ^{\D}$ is not compact on  $L_{(0,1)}^{2}(\D).$ 
\end{theorem}

See Remark \ref{InfiniteDimentionalKernelBox} for an explanation of why we
stated the above theorem for domains in $\C^n$ for $n\geq 3.$ 

\begin{remark}
Hankel operators are closely connected to a very important class of
operators called Toeplitz operators. The Toeplitz operator on $A^2(\D)$ with
symbol $\phi\in L^{\infty}(\D),$ denoted by $T^{\D}_{\phi},$ is  defined as
$T^{\D}_{\phi}f=P^{\D}(\phi f)=\phi f-H^{\D}_{\phi}f$ for $f\in A^2(\D).$   
Let $\phi\in C(\overline{\D})$ such that  $\phi(z)\neq 0$ for
$z\in b\D_1.$ Choose $\psi,\phi_1\in C(\overline{\D})$ such that
$\phi_1(z)=\phi(z)$ for $z\in b\D_1$ and  $|\phi_1|>0$ on 
$\overline{\D},$ and $\psi=1/\phi_1.$ Then 
\[(T^{\D}_{\phi}-T^{\D}_{\phi_1}+T^{\D}_{\phi_1})T^{ \D}_{ \psi}=
T^{\D}_{\phi-\phi_1}T^{\D}_{\psi}+T^{\D}_{\phi_1\psi}
-P^{\D}M_{\phi_1}H^{\D}_{\psi}=I+K \]
where $K=T^{\D}_{\phi-\phi_1}T^{\D}_{\psi}-P^{\D}M_{\phi_1} H^{\D}_{\psi}.$  
Now assume that  $\D,\D_1,$ and $\D_2$ are as in Theorem \ref{ThmMain}. 
Then $K$ is a compact operator ($T^{\D}_{\phi-\phi_1}$ is compact because
$\phi_1-\phi=0$ on the outer boundary of $\D$ and  $H^{\D}_{\psi}$ is  compact
by Theorem \ref{ThmMain}). Hence, $T_{\phi}^{\D}$ is Fredholm for any
$\phi\in C(\overline{\D})$ with the property that $\phi(z)\neq 0$ for
$z\in b\D_1.$ Fredholm property of Toeplitz operators on some pseudoconvex
domains in $\C^n$ has been studied by several authors (see, for example, 
\cite{Venugopalkrishna72,HenkinIordan97}).
\end{remark}

\begin{remark}
Hankel operators can also be expressed as commutators of the Bergman projection
with multiplication operators. These commutators proved to be useful in the
proof of  the complex version of Hilbert's seventeenth problem (see
\cite{CatlinD'Angelo97}). For more information about relations between the
commutators and the $\dbar$-Neumann problem we refer the reader to \cite[Chapter
4.1]{StraubeBook}. The computation is as follows:
\[ \langle P^{\D} (\phi g), h \rangle_{L^2(\D)} = \langle \phi g, h \rangle
_{L^2(\D)}=\langle g, H^{\D}_{\overline{\phi}}h \rangle_{L^2(\D)} = \langle
(H^{\D}_{\overline{\phi}})^*g, h \rangle_{L^2(\D)} \]
for $\phi \in L^{\infty} (\D), h \in A^2(\D),$ and $g \bot A^2(\D).$ Hence for
any $f\in L^2(\D)$ we have
\begin{align*} 
 [M_{\phi}, P^{\D}](f) &=[M_{\phi}, P^{\D}](P^{\D}f) +[M_{\phi},
P^{\D}]((I-P^{\D})f)\\
 &= H^{\D}_{\phi}P^{\D}f-P^{\D}M_{\phi}(I-P^{\D})f\\
 &= H^{\D}_{\phi}P^{\D}f-(H^{\D}_{\overline{\phi}})^*(I-P^{\D})f.
\end{align*} 
When $f\in A^2(\D),P^{\D}f=f,$ and $(I-P^{\D})f=0,$ whence
$H_{\phi}^{\D}=[M_{\phi},P^{\D}]$ on   $A^2(\D).$ Note that
$(H^{\D}_{\overline{\phi}})^*:L^2(\D)\to A^2(\D)$ and it is compact if and only
if $H^{\D}_{\overline{\phi}}$ is compact. Therefore, $H_{\phi}^{\D}$ is compact
on $A^2({\D})$ if and only if $[M_{\phi},P^{\D}]$ is compact on $L^2(\D).$ We
note that similar calculations as well as related issues appeared in
\cite{Haslinger08} on pseudoconvex domains (see also
\cite{CatlinD'Angelo97,FuStraube01}).
\end{remark}

\begin{corollary} \label{Cor}  
Let $\D,\D_1,$ and $\D_2$ be as in Theorem \ref{ThmMain}. Then  the commutator
$[M_{\phi}, P^{\D}]$ is compact on $L^{2}(\D)$  for every  
$\phi \in C(\overline{\D})$ but the $\dbar$-Neumann operator $N ^{\D}$ is not
compact on $L_{(0,1)}^{2}(\D).$ 
\end{corollary}

\begin{example} Here, we give an explicit example. 
Let $\lambda_1(t)= 0$ for $t\leq 0$ and 
$\lambda_1(t)= e^{-1/t}$ for $t>0$ and 
\[\lambda(z_1,z_2,z_3)=\lambda_1\left(|z_1|^2-\frac{1}{4}\right)
+\lambda_1\left(|z_2|^2-\frac{1}{4}\right)+
\lambda_1\left(|z_3|^2-\frac{1}{4}\right)-e^{-3}.\]
One can check that $\lambda_1$ is a convex function on $(-\infty,1/2).$ Let
us define 
\[ \D=\left\{(z_1,z_2,z_3)\in \C^3: |z_1|^2+|z_2|^2+|z_3|^2<9\text{ and }
\lambda(z_1,z_2,z_3)>0 \right\}.\]
So $\D_1= B(0,3),\D_2= \left\{(z_1,z_2,z_3)\in \C^3: \lambda(z_1,z_2,z_3)<0
\right\},$ and $\D=\D_1\setminus \Dc_2$ is a smooth bounded annulus type domain
in $\C^3.$ By construction $b\D\cap B((0,0,\sqrt{7}/\sqrt{12}),1/3)$ is a
Levi-flat surface. Then 
Theorem \ref{ThmMain} and Corollary \ref{Cor} imply that $[M_{\phi}, P^{\D}]$ is
compact on $L^{2}(\D)$ for every $\phi \in C(\overline{\D})$ (hence
$H^{\D}_{\phi}$ is compact on $A^{2}(\D)$ for every $\phi \in C(\overline{\D})$)
but  $N^{\D}$ is not compact on $L^{2}_{(0,1)}(\D).$
\end{example}

\section*{Proof of Theorem \ref{ThmMain}}
Let 
\[\mathcal{H}^{\D}_{(0,1)}=\mbox{ker}(\Box^{\D})=\{u\in\mbox{Dom}(\overline{
\partial})\cap\mbox{Dom}(\overline{\partial}^{*}) :
\overline{\partial}u=0,\ \overline{\partial}^{*}u=0\} \subset 
L^2_{(0,1)}(\D). \]
We call $\mathcal{H}^{\D}_{(0,1)}$ the space of harmonic $(0,1)$-forms and
denote $H^{\D}$ the orthogonal projection from $L^2_{(0,1)}(\D)$ onto
$\mathcal{H}^{\D}_{(0,1)}.$  The following Lemmas will be useful in the
proof of Theorem  \ref{ThmMain}.

\begin{lemma} \label{LemCompEstimate}
Let $\D$ be an annulus type domain in $\C^n$ for $n\geq 2.$  Then $N^{\D}$ is
compact on $L^2_{(0,1)}(\D)$ if and only if for every $\varepsilon>0$ there
exists $C_{\varepsilon}>0$ such that 
\begin{equation}\label{CompEstimate}
\| u\|^2\leq \varepsilon \left(\|\dbar u\|^2 +\|\dbar^*u \| ^2\right)+\| H^{\D}
u\|^2+C_{\varepsilon} \| u-H^{\D}u\|^2_{-1} 
\end{equation}
 for  $u\in Dom(\dbar)\cap Dom(\dbar^*)\subset L^2_{(0,1)}(\D).$
\end{lemma}

\begin{proof} 
We note that $\dbar$ has closed range in $L^2_{(0,1)}(\D)$ 
(see \cite[Theorem 3.3]{Shaw10}). Let us define 
\[\Gamma=Dom(\dbar) \cap
Dom(\dbar^*)\cap \left(\mathcal{H}^{\D}_{(0,1)}\right)^{\bot}\] 
($X^{\bot}$ denotes the orthogonal complement of $X$) and  equip the space
$\Gamma$ with the graph norm. That is, 
$\|u\|^2_{\Gamma}=\|\dbar u\|^2+\|\dbar^* u\|^2.$ Then the embedding  
$j:\Gamma \hookrightarrow L^2_{(0,1)}(\D)$ is continuous \cite{Shaw10}. 
Furthermore,  $N=j\circ j^*$ (for a proof of this see
\cite[Theorem 2.9]{StraubeBook}. Although pseudoconvexity is assumed in 
\cite[Theorem 2.9]{StraubeBook} its proof applies in our situation as well
because $j$ is a bounded operator). Hence $N$
is compact if and only if $j$ is compact and compactness of $j$ is equivalent to
the following estimate (\cite[Proposition 4.2]{StraubeBook}): for
all $\varepsilon>0$ there exists $C_{\varepsilon}>0$ such that 
\[  \|u\|^2 \leq \varepsilon \left(\|\dbar u\|^2+\|\dbar^* u\|^2\right)
+C_{\varepsilon} \|u\|^2_{-1} \mbox{ for } u\in \Gamma. \]
One can substitute $u-H^{\D}u$ instead of $u$ above to show that the inequality
above is equivalent to \eqref{CompEstimate}.
\end{proof}

\begin{lemma} \label{LemLocalization}
Let $\D$ be an annulus type domain in $\C^n$ for $ n\geq 3$ such that
$N^{\D}$ exists and it is compact  on $L^2_{(0,1)}(\D).$ Let $p$ be a boundary
point of $\D$ and $r>0$ such that $U=\D\cap B(p,r)$ is a pseudoconvex domain. 
Then $N^{U}$ is compact on $L^2_{(0,1)}(U).$ 
\end{lemma}
\begin{proof}
 We note that since $n\geq 3$ the space $\mathcal{H}^{\D}_{(0,1)}$ is trivial
and the proof is essentially contained in \cite[Proposition 4.4]{StraubeBook}
once we know that $\mathcal{H}^{\D}_{(0,1)}$ is trivial. However, we will give
the proof here for the convenience of the reader. 

Lemma \ref{LemCompEstimate} implies that
compactness of $N^{\D}$ is equivalent to the following estimate:
for all $\varepsilon>0$ there exists $C_{\varepsilon}>0$ such that 
\[ \|u\|^2_{\D}\leq \varepsilon (\|\dbar u\|^2_{\D}+\|\dbar^*
u\|^2_{\D})+C_{\varepsilon}\|u\|^2_{-1,\D}\ \text{ for }\ u\in Dom(\dbar)\cap
Dom(\dbar^*).\]
Let
$\displaystyle \lambda_{\varepsilon}(z)=\frac{\|z-p\|^2-r^2}{
\varepsilon }.$ One can check  that
$-2r+\varepsilon \leq\lambda_{\varepsilon}(z)\leq 2r+\varepsilon$ and 
\[\sum_{j,k=1}^{n}\frac{\partial^{2}\lambda_{\varepsilon}(z)}{\partial
z_{j}\partial\overline{z}_{k}}\xi_{j}\overline{\xi}_{k}\geq\frac{1}{
\varepsilon}
\|\xi\|^{2}\]
for $z\in V_{\varepsilon}=\lbrace z\in\C^{n}\ \vert\ 
\mbox{dist}(z,bB(p,r))<\varepsilon\rbrace$ for $\xi\in \C^{n}$. 
Now we choose $\phi_{\varepsilon}$ as a smooth cut-off function
$(0\leq\phi_{\varepsilon}\leq 1)$, $\phi_{\varepsilon}\equiv 1$ near $bB(p,r)$
and supported in $V_{\varepsilon}$. The triangle inequality implies that 
\begin{align}\label{ineq2aa}
\norm{u}_{U}^{2}\leq
2\norm{\phi_{\varepsilon}u}_{U}^{2}+2\norm{(1-\phi_{\varepsilon})u}_{U}^
{2} \text{ for } u\in L^2_{(0,1)}(U).
\end{align}
Let $u\in
\mbox{Dom}(\overline{\partial})\cap\mbox{Dom}(\overline{\partial}^{*})\subset 
L^2_{(0,1)}(U)$ then 
$(1-\phi_{\varepsilon})u\in
\mbox{Dom}(\overline{\partial})\cap\mbox{Dom}(\overline{\partial}^{*})\subset
L_{(0,1)}^{2}(\Omega).$ Since the domain $U$ is not $C^2$-smooth a
direct application of  Morrey-Kohn-H\"{o}rmander formula is not possible.
However, one can use the Morrey-Kohn-H\"{o}rmander formula (with
weight $\lambda_{\varepsilon}+\psi$) with  the exhaustion procedure developed in
\cite{Straube97} (see also \cite[Corollary 2.15]{StraubeBook}) together with
the  fact that $\phi_{\varepsilon}u$ belongs to the domain of
$\dbar^*$ on $U$ to show that 
\begin{align}\label{ineq2d}
\norm{\phi_{\varepsilon}u}_{U}^{2}\lesssim \varepsilon(
\norm{\overline{\partial}\left(\phi_{\varepsilon}u\right)}_{U}^{2}+\norm{
\overline{\partial}^*\left(\phi_{\varepsilon}u\right)}_{U}^{2}).
\end{align}
In the inequality above we used generalized constants. That is, $A\lesssim
B$ denotes that $A\leq cB$ where $c>0$  is independent of quantities of
interest. Thus, from \eqref{ineq2aa} and \eqref{ineq2d} we get
\begin{align*}
\norm{u}_{U}^{2}&\lesssim
\varepsilon\left(\norm{\overline{\partial}\left(\phi_{\varepsilon}u\right)}_{
U}^{2}+\norm{\overline{\partial}^*\left(\phi_{\varepsilon}u\right)}_
{U}^{2}\right)+\norm{(1-\phi_{\varepsilon})u}_{U}^{2}\\
 &\lesssim
\varepsilon\left(\norm{\overline{\partial}u}_{U}^{2}+\norm{\overline{
\partial}^*u}_{U}^{2}+\norm{(\nabla\phi_{\varepsilon})u}_{U}^{2
}\right)+\norm{(1-\phi_{\varepsilon})u}_{U}^{2}.
\end{align*}
$(1-\phi_{\varepsilon})u$ and $\nabla\phi_{\varepsilon}u$ can be viewed as forms
on $\Omega$ in
$\mbox{Dom}(\overline{\partial})\cap\mbox{Dom}(\overline{\partial}^*)$. Let
us choose $\chi_{\varepsilon}\in C_{0}^{\infty}(B(p,r))$ such that
$\chi_{\varepsilon}\equiv 1$ on the union of the support of 
$\nabla\phi_{\varepsilon}$ and support of $1-\phi_{\varepsilon}.$ Then
\begin{align}\label{eq1222}
\norm{u}_{U}^{2}\lesssim 
\varepsilon\left(\norm{\overline{\partial}u}_{U}^{2}+\norm{\overline{
\partial}^*u}_{U}^{2}\right)+C_{\phi_{\varepsilon}}\norm{\chi_{
\varepsilon}u}_{U}^{2}
\end{align}
for $ u\in\mbox{Dom}(\overline{\partial})\cap\mbox{Dom} 
(\overline{\partial}^*)\subset L^2_{(0,1)}(U).$ 
Now, we will try to estimate the last term in \eqref{eq1222}.
We note that $\chi_{\varepsilon}
u\in\mbox{Dom}(\overline{\partial})\cap\mbox{Dom}(\overline{\partial}^*)\subset
L_{(0,1)}^{2}(\Omega)$ for any
$u\in\mbox{Dom}(\overline{\partial})\cap\mbox{Dom}(\overline{\partial}^*)\subset
L_{(0,1)}^{2}(U).$ Compactness of $N^{\D}$ on
$L_{(0,1)}^{2}(\D)$ implies that  for all $\varepsilon '>0$ there
exists $C_{\varepsilon '}>0$ such that 
\begin{align}\label{eq133}
\nonumber \norm{\chi_{\varepsilon}u}_{\Omega}^{2}&\leq  \varepsilon '\left(
\norm{\overline{\partial}(\chi_{\varepsilon}u)}_{\Omega}^{2}+\norm{
\overline{\partial}^*(\chi_{\varepsilon}u)}_{\Omega}^{2}\right)+C_{
\varepsilon '}\norm{\chi_{\varepsilon}u}_{-1,\Omega}^{2}\\
 &\leq  \varepsilon '(
\norm{\chi_{\varepsilon}\overline{\partial}u}_{\Omega}^{2} +
\norm{\chi_{\varepsilon}\overline{\partial}^*u}_{\Omega}^{2}+\norm{\nabla\chi_{
\varepsilon
}u}_{\Omega}^{2})+ C_{\varepsilon '}\norm{u}_{-1,U}^{2}
\end{align} 

Now, let us estimate $\norm{\nabla\chi_{\varepsilon}u}_{\Omega}^{2}$ in
(\ref{eq133})
\begin{align*}
\norm{\nabla\chi_{\varepsilon}u}_{\Omega}^{2} \lesssim
\norm{\nabla\chi_{\varepsilon}u}_{U}^{2} 
\nonumber &\lesssim C_{\varepsilon}\norm{u}_{U}^{2} \lesssim
C_{\varepsilon}\left(\norm{\overline{\partial}u}_{U}^{2}+\norm{\overline{
\partial}^*u}_{U}^{2}\right).
\end{align*}
In the last step we used the basic estimate on $U$, $\norm{u}_{U}^{2}\leq
C(\norm{\overline{\partial}u}_{ U}^{2}+ \norm{\overline{\partial}^*u}_{U}).$

Thus, combining \eqref{eq133} and the discussion after it with \eqref{eq1222}
we get  a compactness estimate on $U$. That is, for all $\varepsilon>0$ there
exists $ C_{\varepsilon}>0$ such that 
\begin{align}
\norm{u}_{U}^{2} \lesssim \varepsilon\left(
\norm{\overline{\partial}u}_{U}^{2}+\norm{\overline{\partial}^*u}_{
U}^{2}\right)+C_{\varepsilon}\norm{u}_{-1,U}^{2}
\end{align}
for all $u\in
\mbox{Dom}(\overline{\partial})\cap\mbox{Dom}(\overline{\partial}^*)\subset
L_{(0,1)}^{2}(U)$. 
\end{proof}

\begin{remark}\label{InfiniteDimentionalKernelBox}
We chose the domain $\D$ in $\C^n$ for $n\geq 3$ because we do not know
if the localization in the proof of Lemma \ref{LemLocalization} is possible when
$n=2.$ If $\D\subset \C^2$ is an annulus type domain then
$\mathcal{H}^{\D}_{(0,1)}$ is an infinite dimensional space 
\cite[Theorem 3.5]{Shaw10} whereas $\mathcal{H}^{\D}_{(0,1)}$ is trivial when
$\D$ is pseudoconvex. 
\end{remark}

In the following Lemma $R_{V}$ denotes the restriction operator onto $V.$

\begin{lemma} \label{LemHankelCompact}
Let $\D=\D_1\setminus \Dc_2$ where  $\D_1$ and $\D_2$ are two smooth bounded
pseudoconvex domains in $\C^n, n\geq 3,$ such that $\Dc_2\subset \D_1$ and  
$\phi\in C^1(\Dc_1).$ Then $H^{\D_1}_{\phi}$ is compact on  $A^2(\D_1)$ if and
only if  $H^{\D}_{R_{\D}(\phi)}$ is compact on $A^{2}(\D).$
\end{lemma}

\begin{proof}
Let us prove the necessity first. By Hartogs extension theorem there
exists a unique bounded  extension operator $E_{\D}^{\D_1}:A^2(\D)\to A^2(\D_1).$ 
One can check that $R_{\D}H^{\D_1}_{\phi} E_{\D}^{\D_1}f$ solves   $\dbar u= f \dbar\phi$
on $\D.$ Furthermore, since $H^{\D}_{R_{\D}(\phi)}f$ is the canonical solution
(the solution with minimal $L^2$ norm) for $\dbar u= f \dbar\phi$ we have 
\[(I-P^{\D})R_{\D}H^{\D_1}_{\phi} E_{\D}^{\D_1} =H^{\D}_{R_{\D}(\phi)}.\]
Therefore, compactness of  $H^{\D_1}_{\phi}$  on $A^{2}(\D_1)$ implies that  
 $H^{\D}_{R_{\D}(\phi)}$ is compact on $A^{2}(\D).$ 

To prove the converse assume that $H^{\D}_{R_{\D}(\phi)}$ is compact on
$A^{2}(\D)$ and $U$ be a neighborhood of $p\in b\D_1$ such that 
$U\cap \D_1=U\cap \D$ is a domain. Then i. in
\cite[Proposition 1]{CuckovicSahutoglu09} implies that 
$H^{\D\cap U}_{R_{\D\cap U}(\phi)}R_{\D\cap U}$ is compact on $A^{2}(\D).$ We
note that even though \cite[Proposition 1]{CuckovicSahutoglu09} is stated for
pseudoconvex domains, i.   is still true for general domains. However, one
can check that $(I-P^{\D_1\cap U})H^{\D\cap U}_{R_{\D\cap U}(\phi)}R_{\D\cap
U}=H^{\D_1\cap U}_{R_{\D_1\cap U}(\phi)}R_{\D_1\cap U}$  on $A^2(\D_1)$
and hence  $H^{\D_1\cap U}_{R_{\D_1\cap U}(\phi)}R_{\D_1\cap U}$ is
compact on $A^{2}(\D_1).$ Now ii. \cite[Proposition 1]{CuckovicSahutoglu09}
implies that $H^{\D_1}_{\phi}$ is compact on $A^{2}(\D_1).$ 
\end{proof}

\begin{remark}\label{PreservingCompactnessOfN} We note that compactness of
$N^{\D_{1}}$ on $A_{(0,1)}^{2}(\Omega_{1})$ is equivalent to compactness of
$N^{\D}$ on $A_{(0,1)}^{2}(\D)$. This can be seen as follows:

Range's formula, $N^{\D_{1}}=\left(\dbar^{*}N_{2}^{\D_{1}}\right)\left(\dbar^{*}
N_{2}^{\D_{1}}\right)^{*}+\left(\dbar^{*}N^{\D_{1}}\right)^{*} 
\left(\dbar^{*}N^{\D_{1}}\right),$ together with the fact that 
$\left(\dbar^{*}N_{2}^{\D_{1}}\right)^{*}u=0$ for $u\in A_{(0,1)}^{2}(\D_{1})$
imply that $N^{\D_{1}}u=\left(\dbar^{*}N^{\D_{1}}\right)^{*}\left(\dbar^{*} 
N^{\D_{1}}\right)u $ for $u\in A_{(0,1)}^{2}(\D_{1}).$ Hence, 
$N^{\D_1}|_{A_{(0,1)}^{2}(\D_1)}$ is  compact if and  only if 
$\dbar^{*}N^{\D_1}|_{A_{(0,1)}^{2}(\D_1)}$ is compact. (Here $f|_X$ denotes the
restriction of the operator $f$ onto the space $X$). Similarly, one can show
that  $N^{\D}|_{A_{(0,1)}^{2}(\Omega)}$ is compact if and only if 
$\dbar^{*}N^{\D}|_{A_{(0,1)}^{2}(\Omega)}$ is compact. On the other
hand, Lemma \ref{LemHankelCompact} implies that compactness of
$\dbar^{*}N^{\D_{1}}|_{A_{(0,1)}^{2}(\Omega_{1})}$ is equivalent to  
compactness of $\dbar^{*}N^{\D}|_{A_{(0,1)}^{2}(\D)}.$
\end{remark}

We will need the following theorem of Catlin.  For a proof we refer the reader
to the proof of Proposition 9 in \cite{FuStraube01} 
(see also \cite{SahutogluStraube06}). We note that even though Catlin's Theorem
in \cite{FuStraube01} is stated in  $\C^2,$ the same proof works for the
following version in $\C^n.$

\begin{theorem*}[Catlin] Let $\D$ be a bounded pseudoconvex domain in 
$\C^n, n\geq 2,$ with  Lipschitz boundary. Assume that  $b\D$ contains
an $(n-1)$-dimensional complex manifold. Then  $N^{\D}$ is not compact
on $L^2_{(0,1)}(\D).$  
\end{theorem*}

\begin{proof}[Proof of Theorem \ref{ThmMain}] 
The assumption that $N^{\D_1}$ is compact implies that $H^{\D_1}_{\phi}$ is
compact for all $\phi\in C^1(\Dc_1)$ (see \cite[Proposition 4]{FuStraube01},
\cite[Proposition 4.1]{StraubeBook}, and \cite[Theorem 3]{Haslinger08}). Since
any $\phi\in C^1(\Dc)$ can be extended as $C^1$ function on $\C^n$  
Lemma \ref{LemHankelCompact} implies that $H^{\D}_{\phi}$ is compact for all
$\phi\in C^1(\Dc).$ However, one can approximate any $\phi\in C(\Dc)$ uniformly
on $\Dc$ by $C^1$ functions. Therefore,  we conclude that $H^{\D}_{\phi}$ is
compact on $A^{2}(\D)$ for all $\phi\in C(\Dc).$ 
 
Now we will show that $N^{\D}$ is not compact. Shaw's Theorem, stated
in the introduction, implies that $N^{\D}$ is a bounded operator  on
$L_{(0,1)}^{2}(\D)$. Assume  that $N^{\D}$ is compact on
$L^2_{(0,1)}(\D).$ Let us choose $p\in b\D_2$ and $r>0$ so that  $U=\D\cap
B(p,r)$ is a domain that does not intersect $b\D_1$ and  the (inner) boundary of
$\D$ in $B(p,r)$ is Levi-flat. Hence $U$ is a non-smooth bounded 
pseudoconvex domain. Lemma \ref{LemLocalization} implies that if
$N^{\D}$ is compact on $L_{(0,1)}^{2}(\D)$ then $N^{U}$  is
compact on $L_{(0,1)}^{2}(U).$ Compactness of $N^U$ implies that $\dbar$ has a
compact solution operator on $L^2_{(0,1)}(U).$ However, this contradicts 
Catlin's Theorem stated above.  Hence, $N^{\D}$ is not compact on
$L_{(0,1)}^{2}(\D).$ This contradiction  with the assumption that $N^{\D}$ is
compact completes the proof.  
\end{proof}

\section*{acknowledgement}
We would like to thank our advisor Emil Straube for valuable comments on a
preliminary version of this manuscript.



\begin{thebibliography}{Koh63}

\bibitem[CD97]{CatlinD'Angelo97}
David~W. Catlin and John~P. D'Angelo, \emph{Positivity conditions for
  bihomogeneous polynomials}, Math. Res. Lett. \textbf{4} (1997), no.~4,
  555--567.

\bibitem[CS01]{ChenShawBook}
So-Chin Chen and Mei-Chi Shaw, \emph{Partial differential equations in several
  complex variables}, AMS/IP Studies in Advanced Mathematics, vol.~19, American
  Mathematical Society, Providence, RI, 2001.

\bibitem[{\v{C}}{\c{S}}09]{CuckovicSahutoglu09}
{\v{Z}}eljko {\v{C}}u{\v{c}}kovi{\'c} and S{\"o}nmez {\c{S}}ahuto{\u{g}}lu,
  \emph{Compactness of {H}ankel operators and analytic discs in the boundary of
  pseudoconvex domains}, J. Funct. Anal. \textbf{256} (2009), no.~11,
  3730--3742.

\bibitem[FS01]{FuStraube01}
Siqi Fu and Emil~J. Straube, \emph{Compactness in the
  {$\overline\partial$}-{N}eumann problem}, Complex analysis and geometry
  (Columbus, OH, 1999), Ohio State Univ. Math. Res. Inst. Publ., vol.~9, de
  Gruyter, Berlin, 2001, pp.~141--160.

\bibitem[Has08]{Haslinger08}
Friedrich Haslinger, \emph{The {$\overline\partial$}-{N}eumann operator and
  commutators of the {B}ergman projection and multiplication operators},
  Czechoslovak Math. J. \textbf{58(133)} (2008), no.~4, 1247--1256.

\bibitem[HI97]{HenkinIordan97}
Gennadi~M. Henkin and Andrei Iordan, \emph{Compactness of the {N}eumann
  operator for hyperconvex domains with non-smooth {$B$}-regular boundary},
  Math. Ann. \textbf{307} (1997), no.~1, 151--168.

\bibitem[H{\"o}r65]{Hormander65}
Lars H{\"o}rmander, \emph{{$L\sp{2}$} estimates and existence theorems for the
  {$\bar \partial $}\ operator}, Acta Math. \textbf{113} (1965), 89--152.

\bibitem[Koh63]{Kohn63}
J.~J. Kohn, \emph{Harmonic integrals on strongly pseudo-convex manifolds. {I}},
  Ann. of Math. (2) \textbf{78} (1963), 112--148.

\bibitem[Sha10]{Shaw10}
Mei-Chi Shaw, \emph{The closed range property for $\overline{\partial}$ on
  domains with pseudoconcave boundary}, Complex analysis, Trends Math.,
  Birkh\"auser, Basel, 2010, pp.~307 -- 320.

\bibitem[{\c{S}}S06]{SahutogluStraube06}
S{\"o}nmez {\c{S}}ahuto{\u{g}}lu and Emil~J. Straube, \emph{Analytic discs,
  plurisubharmonic hulls, and non-compactness of the
  {$\overline\partial$}-{N}eumann operator}, Math. Ann. \textbf{334} (2006),
  no.~4, 809--820.

\bibitem[Str97]{Straube97}
Emil~J. Straube, \emph{Plurisubharmonic functions and subellipticity of the
  {$\overline\partial$}-{N}eumann problem on non-smooth domains}, Math. Res.
  Lett. \textbf{4} (1997), no.~4, 459--467.

\bibitem[Str10]{StraubeBook}
\bysame, \emph{Lectures on the $\mathcal{L}^{2}$-{S}obolev theory of the
  $\overline\partial$-{N}eumann problem}, ESI Lectures in Mathematics and
  Physics, vol.~7, European Mathematical Society (EMS), Z\"urich, 2010.

\bibitem[Ven72]{Venugopalkrishna72}
U.~Venugopalkrishna, \emph{Fredholm operators associated with strongly
  pseudoconvex domains in {$C\sp{n}$}}, J. Functional Analysis \textbf{9}
  (1972), 349--373.

\end{thebibliography}
\end{document}